\documentclass[12pt]{article}
\usepackage{amsmath,amsthm,amssymb}
\usepackage{amssymb,latexsym}
\newtheorem{theorem}{Theorem}
\newtheorem{conjecture}{Conjecture}
\newtheorem{lemma}{Lemma}

\begin{document}

\baselineskip=17pt

\title{\bf A logarithmic inequality involving prime numbers}

\author{\bf S. I. Dimitrov}
\date{2019}
\maketitle
\begin{abstract}
Assume that $N$ is a sufficiently large positive number. In this paper we show
that for a small constant $\varepsilon>0$, the logarithmic inequality
\begin{equation*}
\big|p_1\log p_1+p_2\log p_2+p_3\log p_3-N\big|<\varepsilon
\end{equation*}
has a solution in prime numbers $p_1,\,p_2,\,p_3$.
\quad\\
{\bf  2010 Math.\ Subject Classification}:  11P55 $\cdot$ 11J25
\end{abstract}

\section{Introduction and statements of the result}
\indent

One of the most remarkable diophantine inequality with
prime numbers is the ternary Piatetski-Shapiro inequality.
The first solution is due to Tolev \cite{Tolev1}.
In 1992 he considered the the diophantine inequality
\begin{equation}\label{Tolevinequality}
|p_1^c+p_2^c+p_3^c-N|<\varepsilon,
\end{equation}
where $N$ is a sufficiently large positive number,
$p_1,\,p_2,\,p_3$ are prime numbers, $c>1$ is not an integer and $\varepsilon>0$ is a small constant.
Overcoming all difficulties Tolev \cite{Tolev1}
showed that \eqref{Tolevinequality} has a solution for
\begin{equation*}
1<c<\frac{15}{14}.
\end{equation*}
Afterwards the result of Tolev was improved by Cai \cite{Cai1} to
\begin{equation*}
1<c<\frac{13}{12},
\end{equation*}
by Cai \cite{Cai2} and Kumchev and Nedeva \cite{Ku-Ne} to
\begin{equation*}
1<c<\frac{12}{11},
\end{equation*}
by Cao and Zhai \cite{Cao-Zhai} to
\begin{equation*}
1<c<\frac{237}{214},
\end{equation*}
by Kumchev \cite{Kumchev} to
\begin{equation*}
1<c<\frac{61}{55},
\end{equation*}
by Baker and Weingartner \cite{Baker-Weingartner} to
\begin{equation*}
1<c<\frac{10}{9},
\end{equation*}
by Cai \cite{Cai3} to
\begin{equation*}
1<c<\frac{43}{36}
\end{equation*}
and this is the best result up to now.

Inspired by these profound investigations in this paper we
introduce new  diophantine inequality with prime numbers.

Consider  the logarithmic inequality
\begin{equation}\label{Myinequality}
\big|p_1\log p_1+p_2\log p_2+p_3\log p_3-N\big|<\varepsilon,
\end{equation}
where $N$ is a sufficiently large positive number and $\varepsilon>0$ is a small constant.
Having the methods of the aforementioned number theorists we expect that
\eqref{Myinequality} can be solved in primes $p_1,\,p_2,\,p_3$.
Thus we make the first step and prove the following theorem.
\begin{theorem} Let $N$ is a sufficiently large positive number.
Let $X$ is a solution of the equality
\begin{equation*}
N=2X\log(2X/3).
\end{equation*}
Then the logarithmic inequality
\begin{equation*}
\big|p_1\log p_1+p_2\log p_2+p_3\log p_3-N\big|<X^{-\frac{1}{25}}\log^8X
\end{equation*}
is solvable in prime numbers $p_1,\,p_2,\,p_3$.
\end{theorem}
As usual the corresponding binary problem is out of reach of the current state of the mathematics.
In other words we have the following challenge.
\begin{conjecture} Let $N$ is a sufficiently large positive number and $\varepsilon>0$ is a small constant.
Then the logarithmic inequality
\begin{equation*}
\big|p_1\log p_1+p_2\log p_2-N\big|<\varepsilon
\end{equation*}
is solvable in prime numbers $p_1,\,p_2$.
\end{conjecture}

We believe that the future development of  analytic number theory
will lead to the solution of this binary  logarithmic conjecture.

\section{Notations}
\indent

For positive $A$ and $B$ we write $A\asymp B$ instead of $A\ll B\ll A$.
As usual $\mu(n)$ is M\"{o}bius' function, $\tau(n)$
denotes the number of positive divisors of $n$ and $\Lambda(n)$ is von Mangoldt's function.
Moreover $e(y)=e^{2\pi \imath y}$.
We denote by $[y]$ the integer part of $y$.
The letter $p$  with or without subscript will always denote prime number.
Let $N$ be an sufficiently large  positive number.
Let $X$ is a solution of the equality
\begin{equation*}
N=2X\log(2X/3).
\end{equation*}
Denote
\begin{align}
\label{varepsilon}
&\varepsilon=X^{-\frac{1}{25}}\log^8X\,;\\
\label{tau}
&\tau= X^{-\frac{23}{25}}\,;\\
\label{K}
&K=X^{\frac{1}{25}}\log^{-6} X\,;\\
\label{Salpha}
&S(\alpha)=\sum\limits_{X/2<p\leq X} e(\alpha p\log p)\log p\,;\\
\label{Int}
&I(\alpha)=\int\limits_{X/2}^{X}e(\alpha y\log y) \,dy\,.
\end{align}

\newpage

\section{Lemmas}
\indent

\begin{lemma}\label{Fourier}Let $k\in \mathbb{N}$.
There exists a function $\psi(y)$ which is $k$ times continuously differentiable and
such that
\begin{align*}
  &\psi(y)=1\quad\quad\quad\mbox{for }\quad\quad|y|\leq 3\varepsilon/4\,;\\[6pt]
  &0\leq\psi(y)<1\quad\mbox{for}\quad3\varepsilon/4 <|y|< \varepsilon\,;\\[6pt]
  &\psi(y)=0\quad\quad\quad\mbox{for}\quad\quad|y|\geq \varepsilon\,.
\end{align*}
and its Fourier transform
\begin{equation*}
\Psi(x)=\int\limits_{-\infty}^{\infty}\psi(y)e(-xy)dy
\end{equation*}
satisfies the inequality
\begin{equation*}
|\Psi(x)|\leq\min\bigg(\frac{7\varepsilon}{4},\frac{1}{\pi|x|},\frac{1}{\pi |x|}
\bigg(\frac{k}{2\pi |x|\varepsilon/8}\bigg)^k\bigg)\,.
\end{equation*}
\end{lemma}
\begin{proof} This is Lemma 1 of Tolev \cite{Tolev1}.
\end{proof}

\begin{lemma}\label{Exponentpairs}
Let $|f^{(m)}(u)|\asymp YX^{1-m}$  for $1<X<u\leq2X$
and $m=1,2,3,\ldots$\\
Then
\begin{equation*}
\bigg|\sum_{X<n\le 2X}e(f(n))\bigg|
\ll Y^\varkappa X^\lambda +Y^{-1},
\end{equation*}
where $(\varkappa, \lambda)$ is any exponent pair.
\end{lemma}
\begin{proof}
See (\cite{Graham-Kolesnik}, Ch. 3).
\end{proof}

\begin{lemma}\label{Iwaniec-Kowalski}
For any complex numbers $a(n)$ we have
\begin{equation*}
\bigg|\sum_{a<n\le b}a(n)\bigg|^2
\leq\bigg(1+\frac{b-a}{Q}\bigg)\sum_{|q|\leq Q}\bigg(1-\frac{|q|}{Q}\bigg)
\sum_{a<n,\, n+q\leq b}a(n+q)\overline{a(n)},
\end{equation*}
where $Q$ is any positive integer.
\end{lemma}
\begin{proof}
See (\cite{Iwaniec-Kowalski}, Lemma 8.17).
\end{proof}

\begin{lemma}\label{Iest}
Assume that $F(x)$, $G(x)$ are real functions defined in  $[a,b]$,
$|G(x)|\leq H$ for $a\leq x\leq b$ and $G(x)/F'(x)$ is a monotonous function. Set
\begin{equation*}
I=\int\limits_{a}^{b}G(x)e(F(x))dx\,.
\end{equation*}
If $F'(x)\geq h>0$ for all $x\in[a,b]$ or if $F'(x)\leq-h<0$ for all $x\in[a,b]$ then
\begin{equation*}
|I|\ll H/h\,.
\end{equation*}
If $F''(x)\geq h>0$ for all $x\in[a,b]$ or if $F''(x)\leq-h<0$ for all $x\in[a,b]$ then
\begin{equation*}
|I|\ll H/\sqrt h\,.
\end{equation*}
\end{lemma}
\begin{proof} 
See (\cite{Titchmarsh}, p. 71).
\end{proof}

\begin{lemma}\label{sumtau}
We have
\begin{align*}
&\emph{(i)}\quad\quad\quad\quad\sum\limits_{n\leq X}\tau^2(n)\ll X\log^3 X,
\quad\quad\quad\quad\quad\quad\quad\quad\quad\quad\quad\quad\quad\quad\quad\quad\\
&\emph{(ii)}\quad\quad\quad\quad\sum\limits_{n\leq X}\Lambda^2(n)\ll X\log X.
\quad\quad\quad\quad\quad\quad\quad\quad\quad\quad\quad\quad\quad\quad\quad\quad
\end{align*}
\end{lemma}

\begin{lemma}\label{SI}
If $|\alpha|\leq\tau$ then
\begin{equation*}
S(\alpha)=I(\alpha)+\mathcal{O}\Big(Xe^{-(\log X)^{1/5}}\Big)\,.
\end{equation*}
\end{lemma}
\begin{proof}
This lemma is very similar to result of Tolev \cite{Tolev1}.
Inspecting the arguments presented in (\cite{Tolev1}, Lemma 14), (with $T=X^{\frac{1}{2}}$)
the reader will easily see that the proof of Lemma \ref{SI} can be obtained by the same manner.
\end{proof}

\begin{lemma}\label{Thetaest}
We have
\begin{equation*}
\int\limits_{-\infty}^{\infty}I^3(\alpha)e(-N\alpha)\Psi(\alpha)\,d\alpha
\gg\varepsilon\frac{X^2}{\log X}.
\end{equation*}
\end{lemma}
\begin{proof}
Denoting the above integral with $\Theta$, using \eqref{Int}, the definition of $\psi(y)$
and the inverse Fourier transformation formula we obtain
\begin{align}\label{HX}
\Theta&=\int\limits_{X/2}^{X}\int\limits_{X/2}^{X}\int\limits_{X/2}^{X}\int\limits_{-\infty}^{\infty}
e((y_1\log y_1+y_2\log y_2+y_3\log y_3-N)\alpha\big)
\Psi(\alpha)\,d\alpha \,dy_1\,dy_2\,dy_3\nonumber\\
&=\int\limits_{X/2}^{X}\int\limits_{X/2}^{X}\int\limits_{X/2}^{X}
\psi(y_1\log y_1+y_2\log y_2+y_3\log y_3-N)\,dy_1\,dy_2\,dy_3\nonumber\\
&\geq\mathop{\int\limits_{X/2}^{X}\int\limits_{X/2}^{X}\int\limits_{X/2}^{X}}
_{\substack{|y_1\log y_1+y_2\log y_2+y_3\log y_3-N|<3\varepsilon/4}}\,dy_1\,dy_2\,dy_3
\geq\int\limits_{\lambda X}^{\mu X}\int\limits_{\lambda X}^{\mu X}
\left(\int\limits_{\Delta}\,dy_3\right)\,dy_1\,dy_2\,,
\end{align}
where $\lambda$ and $\mu$ are real numbers such that
\begin{equation*}
\frac{1}{2}<\frac{2}{3}<\lambda<\mu<\frac{5}{7}<1
\end{equation*}
and
\begin{equation*}
\Delta=\big[X/2,X\big]
\cap\big[y'_3, y''_3\big]=\big[y'_3, y''_3\big],
\end{equation*}
where the interval $\big[y'_3, y''_3\big]$ is a solution of the inequality
\begin{equation*}
N-3\varepsilon/4-y_1\log y_1-y_2\log y_2
<y_3\log y_3<N+3\varepsilon/4-y_1\log y_1-y_2\log y_2.
\end{equation*}
Let $y$ be an implicit function of $t$ defined by
\begin{equation}\label{Implicitfunction}
y\log y=t,
\end{equation}
where
\begin{equation}\label{tasymp}
t\asymp X\log X
\end{equation}
and therefore
\begin{equation}\label{yest}
y\asymp X.
\end{equation}
The first derivative of $y$ is
\begin{equation}\label{Firstderivative}
y'=\frac{1}{1+\log y}.
\end{equation}
By \eqref{yest} and \eqref{Firstderivative} we conclude
\begin{equation}\label{y'est}
y'\asymp\frac{1}{\log X}.
\end{equation}
Thus by the mean-value theorem we get
\begin{equation}\label{Thetaest1}
\Theta\gg\varepsilon\int\limits_{\lambda X}^{\mu X}\int\limits_{\lambda X}^{\mu X}
y'\big(\xi_{y_1,y_2}\big)\,dy_1\,dy_2\,,
\end{equation}
where
\begin{equation*}
\xi_{y_1,y_2}\asymp X\log X.
\end{equation*}
From \eqref{y'est} and \eqref{Thetaest1} it follows that
\begin{equation*}
\Theta\gg\varepsilon\frac{X^2}{\log X}.
\end{equation*}
The lemma is proved.
\end{proof}

\begin{lemma}\label{3Int}
We have
\begin{align*}
&\emph{(i)}\quad\quad\quad\quad\int\limits_{-\tau}^\tau|S(\alpha)|^2\,d\alpha\ll X\log^2X,
\quad\quad\quad\quad\quad\quad\quad\quad\quad\quad\quad\quad\quad\quad\quad\quad\\
&\emph{(ii)}\quad\quad\quad\quad\int\limits_{-\tau}^\tau|I(\alpha)|^2\,d\alpha\ll X,
\quad\quad\quad\quad\quad\quad\quad\quad\quad\quad\quad\quad\quad\quad\quad\quad\\
&\emph{(iii)}\quad\quad\quad\quad\int\limits_{n}^{n+1}|S(\alpha)|^2\,d\alpha\ll X\log^2X.
\quad\quad\quad\quad\quad\quad\quad\quad\quad\quad\quad\quad\quad\quad\quad\quad
\end{align*}
\end{lemma}

\begin{proof}
We only prove  $\textmd{(i)}$. The cases $\textmd{(ii)}$ and $\textmd{(iii)}$  are analogous.

Having in mind  \eqref{Salpha}  we write
\begin{align}\label{Squareout}
\int\limits_{-\tau}^\tau|S(\alpha)|^2\,d\alpha&=\sum\limits_{X/2<p_1,p_2\leq X}\log p_1\log p_2
\int\limits_{-\tau}^\tau e((p_1\log p_1-p_2\log p_2))\alpha)\,d\alpha\nonumber\\
&\ll\sum\limits_{X/2<p_1,p_2\leq X}\log p_1\log p_2
\min\bigg(\tau,\frac{1}{|p_1\log p_1-p_2\log p_2|}\bigg)\nonumber\\
&\ll\tau\sum\limits_{X/2<p_1,p_2\leq  X\atop{|p_1\log p_1-p_2\log p_2|\leq1/\tau}}\log p_1\log p_2\nonumber\\
&+\sum\limits_{X/2<p_1,p_2\leq X\atop{|p_1\log p_1-p_2\log p_2|>1/\tau}}
\frac{\log p_1\log p_2}{|p_1\log p_1-p_2\log p_2|}\nonumber\\
&\ll U\tau\log^2X+V\log^2X,
\end{align}
where
\begin{align*}
&U=\sum\limits_{X/2<n_1,n_2\leq X\atop{|n_1\log n_1-n_2\log n_2|\leq1/\tau}}1\,,\\
&V=\sum\limits_{X/2<n_1,n_2\leq X\atop{|n_1\log n_1-n_2\log n_2|>1/\tau}}
\frac{1}{|n_1\log n_1-n_2\log n_2|}\,.
\end{align*}
On the one hand by the mean-value theorem we get
\begin{equation*}
U\ll\mathop{\sum\limits_{X/2<n_1\leq X}\sum\limits_{X/2<n_2\leq X}1}_
{n_1\log n_1-1/\tau\leq n_2\log n_2\leq n_1\log n_1+1/\tau}
\ll\frac{1}{\tau}\sum\limits_{X/2<n_\leq X}y'(\xi),
\end{equation*}
where $y$ is implicit function defined by the equation \eqref{Implicitfunction}
and $\xi$  satisfies \eqref{tasymp}.
Bearing in mind \eqref{Firstderivative}, \eqref{y'est} and the last inequality we find
 \begin{equation}\label{Uest2}
 U\ll\frac{1}{\tau}\cdot\frac{X}{\log X}\,.
\end{equation}
On the other hand
\begin{equation}\label{VVl}
V\leq\sum\limits_{l}V_l\,,
\end{equation}
where
\begin{equation*}
V_l=\sum\limits_{X/2<n_1,n_2\leq X\atop{l<|n_1\log n_1-n_2\log n_2|\leq2l}}
\frac{1}{|n_1\log n_1-n_2\log n_2|}
\end{equation*}
and $l$ takes the values $2^k/\tau,\,k=0,1,2,...$, with $l\leq X\log X$.\\

Using \eqref{Implicitfunction} -- \eqref{y'est} and  the mean-value theorem we conclude
\begin{align}\label{Vlest}
V_l&\ll\frac{1}{l}\mathop{\sum\limits_{X/2<n_1\leq X}\sum\limits_{X/2<n_2\leq X}}_
{n_1\log n_1+l< n_2\log n_2\leq n_1\log n_1+2l}1\nonumber\\
&+\frac{1}{l}\mathop{\sum\limits_{X/2<n_1\leq X}\sum\limits_{X/2<n_2\leq X}}_
{n_1\log n_1-2l< n_2\log n_2\leq n_1\log n_1-l}1\nonumber\\
&\ll\sum\limits_{X/2<n\leq X}y'(\xi)\nonumber\\
&\ll \frac{X}{\log X}\,.
\end{align}
The proof follows from \eqref{Squareout} -- \eqref{Vlest}.
\end{proof}
\begin{lemma}\label{Salphaest} Assume that $\tau\leq|\alpha| \leq K$.
Then
\begin{equation}\label{Salpha2526}
|S(\alpha)|\ll X^{24/25}\log^3X.
\end{equation}

\begin{proof}
Without loss of generality we may assume that $\tau\leq\alpha \leq K$.

From \eqref{Salpha} we have
\begin{equation}\label{SalphavonMangolt}
S(\alpha)=\sum\limits_{X/2<n\leq X}\Lambda(n)
e(\alpha n\log n)+\mathcal{O}( X^{1/2}).
\end{equation}
On the other hand
\begin{align}\label{SalphaS1alpha}
&\sum\limits_{X/2<n\leq X}\Lambda(n)e(\alpha n\log n)\nonumber\\
&=\sum\limits_{X/2<n\leq X}\Lambda(n)
e\Big(\alpha n\Big(\ \log(n+1)+ \mathcal{O}\big(1/n\big) \Big)  \Big)
\nonumber\\
&=e\big(\mathcal{O}(|\alpha|)\big)\sum\limits_{X/2<n\leq X}
\Lambda(n)e\big(\alpha n\log(n+1)\big)\nonumber\\
&\ll|S_1(\alpha)|,
\end{align}
where
\begin{equation}\label{S1alpha}
S_1(\alpha)=\sum\limits_{X/2<n\leq X}\Lambda(n)e\big(\alpha n\log(n+1)\big).
\end{equation}
We denote
\begin{equation}\label{fdl}
f(d,l)=\alpha dl \log(dl+1).
\end{equation}
Using \eqref{S1alpha}, \eqref{fdl} and Vaughan's identity (see \cite{Vaughan}) we get
\begin{equation}\label{S1alphadecomp}
S_1(\alpha)=U_1-U_2-U_3-U_4,
\end{equation}
where
\begin{align}
\label{U1}
&U_1=\sum_{d\le X^{1/3}}\mu(d)\sum_{X/2d<l\le X/d}(\log l)e(f(d,l)),\\
\label{U2}
&U_2=\sum_{d\le X^{1/3}}c(d)\sum_{X/2d<l\le X/d}e(f(d,l)),\\
\label{U3}
&U_3=\sum_{X^{1/3}<d\le X^{2/3}}c(d)\sum_{X/2d<l\le X/d}e(f(d,l)),\\
\label{U4}
&U_4= \mathop{\sum\sum}_{\substack{X/2<dl\le X \\d>X^{1/3},\,l>X^{1/3} }}
a(d)\Lambda(l) e(f(d,l)),
\end{align}
and where
\begin{equation}\label{cdad}
|c(d)|\leq\log d,\quad  | a(d)|\leq\tau(d).
\end{equation}

\textbf{Estimation of $U_1$ and $U_2$}

Consider first $U_2$ defined by \eqref{U2}.
Bearing in mind  \eqref{fdl} we find
\begin{equation}\label{fnthderivative}
\frac{\partial^nf(d,l)}{\partial l^n}
=(-1)^n\left[\frac{\alpha d^n(n-2)!}{(dl+1)^{n-1}}
+ \frac{\alpha d^n(n-1)!}{(dl+1)^{n}}\right], \quad  \mbox{ for } n\geq2.
\end{equation}
By \eqref{fnthderivative} and the restriction and the restriction
\begin{equation}\label{restriction}
X/2<dl\le X
\end{equation}
we deduce
\begin{equation}\label{fnthderivativeest}
\left|\frac{\partial^nf(d,l)}{\partial l^n}\right|\asymp\alpha d(Xd^{-1})^{1-n}, \quad  \mbox{ for } n\geq2.
\end{equation}
From \eqref{fnthderivativeest} and Lemma \ref{Exponentpairs}
with $(\varkappa, \lambda)=\big(\frac{1}{2}, \frac{1}{2}\big)$  it follows
\begin{align}\label{sumefdl}
\sum\limits_{X/2d<l\le X/d}e(f(d, l))
&\ll(\alpha d)^{1/2}\big(X d^{-1}\big)^{1/2}+(\alpha d)^{-1}\nonumber\\
&=\alpha^{1/2}X^{1/2} +\alpha^{-1} d^{-1}.
\end{align}
Now \eqref{tau}, \eqref{K}, \eqref{U2}, \eqref{cdad} and \eqref{sumefdl} give us
\begin{align}\label{U2est}
U_2&\ll\big(\alpha^{1/2}X^{5/6}+\alpha^{-1/2}\big)\log^2X\nonumber\\
&\ll\big(K^{1/2}X^{5/6}+\tau^{-1/2}\big)\log^2X\nonumber\\
&\ll X^{23/25}\log^2X.
\end{align}
In order to estimate $U_1$ defined by \eqref{U1}
we apply Abel's transformation.
Then arguing as in the estimation of $U_2$  we obtain
\begin{equation}\label{U1est}
U_1\ll X^{23/25}\log^2X.
\end{equation}

\textbf{Estimation of $U_3$ and $U_4$}

Consider first $U_4$ defined by \eqref{U4}. We have
\begin{equation}\label{U4U5}
U_4\ll|U_5|\log X,
\end{equation}
where
\begin{equation}\label{U5}
U_5=\sum_{D<d\le 2D}a(d)\sum_{L<l\le 2L\atop{X/2<dl\leq X}}\Lambda(l)e(f(d,l))
\end{equation}
and where
\begin{equation}\label{ParU5}
X^{1/3}\ll L\ll X^{1/2}\ll D\ll X^{2/3}, \quad  DL\asymp X.
\end{equation}
Using \eqref{cdad}, \eqref{U5}, \eqref{ParU5}, Lemma \ref{sumtau} $\textmd{(i)}$
 and Cauchy's inequality we obtain
\begin{align}\label{U52est1}
|U_5|^2&\ll \sum_{D<d\le 2D}\tau^2(d)\sum_{D<d\le 2D}\bigg|\sum_{L_1<l\le L_2}\Lambda(l)e(f(d,l))\bigg|^2\nonumber\\
&\ll D(\log X)^3\sum_{D<d\le 2D}\bigg|\sum_{L_1<l\le L_2}\Lambda(l)e(f(d,l))\bigg|^2,
\end{align}
where
\begin{equation}\label{maxmin1}
L_1=\max{\bigg\{L,\frac{X}{2d}\bigg\}},\quad
L_2=\min{\bigg\{2L, \frac{X}{d}\bigg\}}\,.
\end{equation}
Now from \eqref{ParU5} -- \eqref{maxmin1}  and Lemma \ref{Iwaniec-Kowalski}
with $Q\leq L$  and Lemma \ref{sumtau} $\textmd{(ii)}$ we find
\begin{align}\label{U52est2}
|U_5|^2&\ll D(\log X)^3   \sum_{D<d\le 2D}\frac{L}{Q}
\sum_{|q|\leq Q}\bigg(1-\frac{|q|}{Q}\bigg)\nonumber\\
&\times\sum_{L_1<l\le L_2\atop{L_1<l+q\le L_2}}\Lambda(l+q)\Lambda(l)
e(f(d,l+q)-f(d,l))\nonumber\\
&\ll \Bigg(\frac{LD}{Q}\sum_{0<|q|\leq Q}
\sum_{L<l\le 2L\atop{L<l+q\le 2L}}\Lambda(l+q)\Lambda(l)
\bigg|\sum_{D_1<d\le D_2}e\big(g_{l,q}(d)\big)\bigg|\nonumber\\
&\quad\quad\quad\quad\quad\quad\quad\quad\quad
\quad\quad\quad\quad\quad\quad\quad\quad\quad+\frac{(LD)^2}{Q}\log X\Bigg)\log^3X,
\end{align}
where
\begin{equation}\label{maxmin2}
D_1=\max{\bigg\{D,\frac{X}{2l},\frac{X}{2(l+q)}\bigg\}},\quad
D_2=\min{\bigg\{2D,\frac{X}{l},\frac{X}{l+q}\bigg\}}
\end{equation}
and
\begin{equation}\label{gd}
g(d)=g_{l,q}(d)=f(d,l+q)-f(d,l).
\end{equation}
It is not hard to see that the sum over negative $q$ in formula \eqref{U52est2}
is equal to the sum over positive $q$. Thus
\begin{align}\label{U52est3}
|U_5|^2\ll\Bigg(\frac{LD}{Q}\sum_{1\leq q\leq Q}
\sum_{L<l\le 2L-q}\Lambda(l+q)\Lambda(l)
\bigg|&\sum_{D_1<d\le D_2}e(g_{l,q}(d))\bigg|\nonumber\\
&\quad\quad+\frac{ (LD)^2}{Q}\log X\Bigg)\log^3X.
\end{align}
Consider the function $g(d)$.
Taking into account   \eqref{fdl}, \eqref{fnthderivative} and \eqref{gd} we get
\begin{equation}\label{gnthderivative}
g^{(n)}(d)=(-1)^n\left[\frac{\alpha (l+q)^n(n-2)!}{(d(l+q)+1)^{n-1}}
+ \frac{\alpha l^n(n-1)!}{(dl+1)^{n}}\right], \quad  \mbox{ for } n\geq2.
\end{equation}
By \eqref{restriction} and \eqref{gnthderivative}  we conclude 
\begin{equation}\label{gnthderivativeest}
|g^{(n)}(d)|\asymp\alpha L(XL^{-1})^{1-n}.
\end{equation}
From \eqref{maxmin2}, \eqref{gnthderivativeest} and Lemma \ref{Exponentpairs}
with $(\varkappa, \lambda)=\big(\frac{11}{82}, \frac{57}{82}\big)$ we obtain
\begin{align}\label{sumegd}
\sum\limits_{D_1<d\leq D_2}e(g(d))
&\ll (\alpha L)^{11/82}\big(X L^{-1}\big)^{57/82}+(\alpha L)^{-1}\nonumber\\
&=\alpha^{11/82}L^{-46/82}X^{57/82}  +\alpha^{-1} L^{-1}.
\end{align}
Bearing in mind \eqref{U52est3}, \eqref{sumegd},
Lemma \ref{sumtau} $\textmd{(ii)}$  and choosing $Q=L$   we find
\begin{align}\label{U52est4}
|U_5|^2&\ll \big( \alpha^{11/82}DLL^{36/82}X^{57/82}
+\alpha^{-1} DL+D^2L\big)\log^4X\nonumber\\
&\ll \big( K^{11/82}L^{36/82}X^{139/82}+\tau^{-1}X+D^2L\big)\log^4X.
\end{align}
Now \eqref{tau}, \eqref{K}, \eqref{ParU5} and \eqref{U52est4} give us
\begin{equation}\label{U5est}
|U_5|\ll \big( K^{11/164}L^{36/164}X^{139/164}+\tau^{-1/2}X^{1/2}+XL^{-1/2}\big)\log^2X
\ll X^{24/25}\log^2X.
\end{equation}
From  \eqref{U4U5} and \eqref{U5est} it follows
\begin{equation}\label{U4est}
U_4\ll X^{24/25}\log^3X.
\end{equation}
Working as in the estimation of $U_4$  we obtain
\begin{equation}\label{U3est}
U_3\ll X^{24/25}\log^3X.
\end{equation}
Summarizing \eqref{SalphavonMangolt}, \eqref{SalphaS1alpha},
\eqref{S1alphadecomp}, \eqref{U2est}, \eqref{U1est}, \eqref{U4est}
and \eqref{U3est} we establish the estimation \eqref{Salpha2526}.

The lemma is proved.
\end{proof}
\end{lemma}

\section{Proof of the Theorem}
\indent

Consider the sum
\begin{equation*}
\Gamma= \sum\limits_{X/2<p_1,p_2,p_3\leq X
\atop{|p_1\log p_1+p_2\log p_2+p_3\log p_3-N|
<\varepsilon}}\log p_1\log p_2\log p_3.
\end{equation*}
The theorem will be proved if we show that $\Gamma\rightarrow\infty$ as $X\rightarrow\infty$.

According to the definition of $\psi(y)$ and the inverse Fourier transformation formula we have
\begin{align}\label{GammaGamma0}
\Gamma&\geq\Gamma_0
= \sum\limits_{X/2<p_1,p_2,p_3\leq X}
\psi(p_1\log p_1+p_2\log p_2+p_3\log p_3-N)\log p_1\log p_2\log p_3\nonumber\\
&= \sum\limits_{X/2<p_1,p_2,p_3\leq X}\log p_1\log p_2\log p_3\nonumber\\
&\times\int\limits_{-\infty}^{\infty}e((p_1\log p_1+p_2\log p_2+p_3\log p_3-N)\alpha)\Psi(\alpha)\,d\alpha \nonumber\\
&=\int\limits_{-\infty}^{\infty}S^3(\alpha)e(-N\alpha)\Psi(\alpha)\,d\alpha.
\end{align}
We decompose $\Gamma_0$ in three parts
\begin{equation}\label{Gamma0decomp}
\Gamma_0=\Gamma_1+\Gamma_2+\Gamma_3.
\end{equation}
where
\begin{align}
\label{Gamma1}
&\Gamma_1=\int\limits_{-\tau}^{\tau}S^3(\alpha)e(-N\alpha)\Psi(\alpha)\,d\alpha,\\
\label{Gamma2}
&\Gamma_2=\int\limits_{\tau\leq|\alpha|\leq K}S^3(\alpha)e(-N\alpha)\Psi(\alpha)\,d\alpha,\\
\label{Gamma3}
&\Gamma_3=\int\limits_{|\alpha|>K}S^3(\alpha)e(-N\alpha)\Psi(\alpha)\,d\alpha.
\end{align}

\textbf{Estimation of $\Gamma_1$}

Denote the integrals
\begin{align}
\label{Thetataudef}
&\Theta_\tau=\int\limits_{-\tau}^\tau I^3(\alpha) e(-N\alpha) \Psi(\alpha)\,d\alpha\\
\label{Thetadef}
&\Theta=\int\limits_{-\infty}^\infty I^3(\alpha) e(-N\alpha) \Psi(\alpha)\,d\alpha.
\end{align}
For $\Gamma_1$ denoted by  \eqref{Gamma1} we have
\begin{equation}\label{Gamma1ThetaThetatau}
\Gamma_1=(\Gamma_1-\Theta_\tau)+(\Theta_\tau-\Theta)+\Theta.
\end{equation}
From \eqref{Gamma1}, \eqref{Thetataudef}, Lemma \ref{Fourier}, Lemma \ref{SI}
and Lemma \ref{3Int} $\textmd{(i)}$, \ref{3Int} $\textmd{(ii)}$ we get
\begin{align}\label{Gamma1Thetatau}
|\Gamma_1-\Theta_\tau|&\ll\int\limits_{-\tau}^{\tau}|S^3(\alpha)-I^3(\alpha)||\Psi(\alpha)|\,d\alpha\nonumber\\
&\ll\varepsilon\int\limits_{-\tau}^{\tau}|S(\alpha)-I(\alpha)|\big(|S(\alpha)|^2+|I(\alpha)|^2\big)\,d\alpha\nonumber\\
&\ll\varepsilon Xe^{-(\log X)^{1/5}}
\Bigg(\int\limits_{-\tau}^{\tau}|S(\alpha)|^2\,d\alpha+\int\limits_{-\tau}^{\tau}|I(\alpha)|^2)\bigg)\nonumber\\
&\ll\varepsilon X^2e^{-(\log X)^{1/6}}.
\end{align}
Using \eqref{tau}, \eqref{Int}, \eqref{Thetataudef}, \eqref{Thetadef}, Lemma \ref{Fourier}
and Lemma \ref{Iest} we find
\begin{align}\label{ThetatauTheta}
|\Theta_\tau-\Theta|&\ll\int\limits_{\tau}^{\infty}|I(\alpha)|^3|\Psi(\alpha)|\,d\alpha
\ll\frac{\varepsilon}{(1+\log X)^3}\int\limits_{\tau}^{\infty}\frac{d\alpha}{\alpha^3}\nonumber\\
&\ll \frac{\varepsilon}{\tau^2(1+\log X)^3}\ll\varepsilon\frac{X^2}{\log^2X}.
\end{align}
Bearing in mind \eqref{tau}, \eqref{Thetadef}, \eqref{Gamma1ThetaThetatau}, \eqref{Gamma1Thetatau},
\eqref{ThetatauTheta} and Lemma \ref{Thetaest} we conclude
\begin{equation}\label{Gamma1est}
\Gamma_1\gg \varepsilon\frac{X^2}{\log X}.
\end{equation}

\textbf{Estimation of $\Gamma_2$}

Now let us consider $\Gamma_2$ defined by \eqref{Gamma2}.
We have
\begin{equation}\label{Gamma2est1}
\Gamma_2\ll\int\limits_{\tau}^{K}|S(\alpha)|^3|\Psi(\alpha)|\,d\alpha
\ll\max\limits_{\tau\leq t\leq K}|S(\alpha)|\int\limits_{\tau}^{K}|S(\alpha)|^2|\Psi(\alpha)|\,d\alpha.
\end{equation}
Using Lemma \ref{Fourier} and Lemma \ref{3Int} $\textmd{(iii)}$ we deduce
\begin{align}\label{InttauK}
\int\limits_{\tau}^{K}|S(\alpha)|^2|\Psi(\alpha)|\,d\alpha
&\ll\varepsilon\int\limits_{\tau}^{1/\varepsilon}|S(\alpha)|^2\,d\alpha
+\int\limits_{1/\varepsilon}^{K}|S(\alpha)|^2\,\frac{d\alpha}{\alpha}\nonumber\\
&\ll\varepsilon\sum\limits_{0\leq n\leq 1/\varepsilon}\int\limits_{n}^{n+1}|S(\alpha)|^2\,d\alpha
+\sum\limits_{1/\varepsilon-1\leq n\leq K}\frac{1}{n}\int\limits_{n}^{n+1}|S(\alpha)|^2\,d\alpha\nonumber\\
&\ll X \log^3X.
\end{align}
From \eqref{Gamma2est1}, \eqref{InttauK} and Lemma \ref{Salphaest} it follows
\begin{equation}\label{Gamma2est2}
\Gamma_2\ll X^{49/25}\log^6X\ll \frac{\varepsilon X^2}{\log^2X}.
\end{equation}

\textbf{Estimation of $\Gamma_3$}

Using (\ref{Gamma3}), Lemma \ref{Fourier} and choosing $k=[\log X]$ we find
\begin{align}\label{Gama3est}
\Gamma_3&\ll \int\limits_{K}^{\infty}|S(\alpha)|^3|\Psi(\alpha)|\,d\alpha\nonumber\\
&\ll X^3\int\limits_{K}^{\infty}\frac{1}{\alpha}\bigg(\frac{k}{2\pi \alpha\varepsilon/8}\bigg)^k\,d\alpha\nonumber\\
&=X^3\bigg(\frac{4k}{\pi\varepsilon K}\bigg)^k\ll1.
\end{align}

\textbf{The end of the proof}

Bearing in mind \eqref{GammaGamma0}, \eqref{Gamma0decomp}, \eqref{Gamma1est}, \eqref{Gamma2est2} and
\eqref{Gama3est} we establish that
\begin{equation}\label{Gammaest}
\Gamma\gg \varepsilon\frac{X^2}{\log X}.
\end{equation}
Now \eqref{varepsilon} and \eqref{Gammaest}
imply that $\Gamma\rightarrow\infty$ as $X\rightarrow\infty$.

The proof of the Theorem is complete.

\vspace{5mm}

\textbf{Acknowledgments.}
The author thanks Professor Kaisa Matom\"{a}ki and Professor Joni Ter\"{a}v\"{a}inen
for their valuable remarks and useful discussions.

\vskip20pt
\footnotesize
\begin{flushleft}
S. I. Dimitrov\\
Faculty of Applied Mathematics and Informatics\\
Technical University of Sofia \\
8, St.Kliment Ohridski Blvd. \\
1756 Sofia, BULGARIA\\
e-mail: sdimitrov@tu-sofia.bg\\
\end{flushleft}

\end{document}